\documentclass[11pt]{article}
\usepackage{latexsym,color,amsmath,amsthm,amssymb,amscd,amsfonts,stmaryrd,mathabx,wasysym}
\usepackage{tikz}
\usepackage{hyperref}

\usepackage{pgfplots}
\usepackage{cite} 
\usetikzlibrary{arrows,calc}
\tikzset{
>=stealth',
help lines/.style={dashed, thick}, axis/.style={<->}, important
line/.style={thick}, connection/.style={thick, dotted}, }

\newtheorem{lemma}{Lemma}[section]
\newtheorem{proposition}[lemma]{Proposition}
\newtheorem{remark}[lemma]{Remark}
\newtheorem{theorem}[lemma]{Theorem}

\newtheorem{corollary}[lemma]{Corollary}

\newtheorem{ques}[lemma]{Question}

\newtheorem*{remark*}{Remark}

\begin{document}
\title {A remark on projections of the rotated cube to complex lines}
\author{Efim D. Gluskin and Yaron Ostrover}
\date{}
\maketitle
\begin{abstract}  
Motivated by relations with a symplectic invariant known as the ``cylindrical symplectic capacity", in this note we study the expectation of the area of a minimal projection to a complex line for a randomly rotated cube.
\end{abstract}

\section{Introduction and Result} \label{sec:Int}

Consider the complex vector space ${\mathbb C}^n$ with coordinates $z = (z_1,\ldots,z_n)$, and equipped with its standard Hermitian structure $\langle z , w \rangle_{\mathbb C} = \sum_{j=1}^n z_j \overline w_j$. 
By writing $z_j = x_j + iy_j$, we can look at ${\mathbb C}^n$ as a real $2n$-dimensional vector space ${\mathbb C}^n \simeq {\mathbb R}^{2n} = {\mathbb R}^n \oplus {\mathbb R}^n$  equipped with the usual complex structure $J$, i.e., $J$ is the linear map $J : {\mathbb R}^{2n} \rightarrow {\mathbb R}^{2n}$  given by $J(x_j,y_j) = (-y_j,x_j)$. 
Moreover, note that the real part  of the Hermitian inner product $\langle \cdot , \cdot \rangle_{\mathbb C}$ is just the standard inner product on ${\mathbb R}^{2n}$, and the imaginary part is the standard symplectic structure on ${\mathbb R}^{2n}$.
As usual, we denote the orthogonal and symplectic groups associated with these two structures by ${\rm O}(2n)$ and ${\rm Sp}(2n)$, respectively. It is well known that ${\rm O}(2n) \cap {\rm Sp}(2n) = {\rm U}(n)$, where the unitary group $ {\rm U}(n)$ is the subgroup of ${\rm GL}(n,{\mathbb C})$ that preserves the above Hermitian inner product. 

Symplectic capacities on ${\mathbb R}^{2n}$ are  numerical invariants which associate with every open set ${\mathcal U} \subseteq {\mathbb R}^{2n}$ a number $c({\mathcal U}) \in [0,\infty]$. This number, roughly speaking, measures the symplectic size of the set ${\mathcal U}$ (see e.g.~\cite{CHLS}, for a survey on symplectic capacities). We refer the reader to the appendix of this paper for more information regarding symplectic capacities, 
and their role as an incentive for the current paper. 
Recently, the authors observed (see Theorem 1.8 in~\cite{GO}) that for symmetric convex domains in ${\mathbb R}^{2n}$,  a certain symplectic capacity $\overline c$, which is the largest possible normalized symplectic capacity and is known as the ``cylindrical capacity", is asymptotically  equivalent to its linearized version given by 
\begin{equation} \label{def-linearized-cilindrical-capacity} \overline c_{_{{\rm Sp}(2n)}} ({\mathcal U})  = \inf_{S \in {\rm Isp}(2n)} {\rm Area} \bigl (\pi( S( {\mathcal U})) \bigr ). \end{equation}
Here, $\pi$ is the orthogonal projection to the complex line $E = \{ z \in {\mathbb C}^n \, | \,  z_j = 0 \ {\rm for} \ j \neq 1  \}$, and the infimum is taken over all $S$ in the affine symplectic group ${\rm ISp}(2n) ={\rm Sp}(2n) \ltimes {\rm T}(2n)$, which is the semi-direct product   of the linear symplectic
group and the group of translations  in ${\mathbb R}^{2n}$. We remark that in what follows we consider only centrally symmetric convex bodies in ${\mathbb R}^{2n}$, and hence one can take $S$ in~$(\ref{def-linearized-cilindrical-capacity})$ to be a genuine symplectic matrix (i.e., $S \in {\rm Sp}(2n)$).

An interesting natural variation of the quantity  $ \overline c_{_{{\rm Sp}(2n)}}$, which serves as an upper bound to it and is of independent interest, is obtained by restricting the infimum on the right-hand side of~$(\ref{def-linearized-cilindrical-capacity})$ to the unitary group ${\rm U}(n)$ (see the Appendix for more details). More preciesly, let $L \subset {\mathbb R}^{2n}$ be a complex line, i.e., $L={\rm span}\{v,Jv\}$ for some non-zero vector $v \in {\mathbb R}^{2n}$, and denote by $\pi_L$  the orthogonal projection to the subspace $L$.  For a symmetric convex body $K \subset {\mathbb R}^{2n}$, the quantity of interest is
\begin{equation} \label{def-linearized-cilindrical-capacity-unitary}  \overline c_{_{{\rm U}(n)}} (K)  :=  \inf_{U \in {\rm U}(n)} {\rm Area} \bigl (\pi( U( {K})) \bigr ) =  \inf \Bigl \{ {\rm Area} \bigl (\pi_L( K) \bigr ) \ |  \  L \subset {\mathbb R}^{2n} \ {\rm is \ a \ complex \ line  }  \Bigr \}.   \end{equation}

In this note we focus on understanding $ \overline c_{_{{\rm U}(n)}}(OQ)$, where $O \in {\rm O}(2n)$ is a random orthogonal transformation, and $Q = [-1,1]^{2n} \subseteq {\mathbb R}^{2n}$ is the standard cube. 
We remark that in~\cite{GO} it was shown that, in contrast with projections to arbitrary two-dimensional subspaces of ${\mathbb R}^{2n}$,  there exist an orthogonal transformation $ O \in {\rm O}(2n)$ such that for every complex line $L \subset {\mathbb R}^{2n}$ one has that
${\rm Area}(\pi_L( O Q)) \geq \sqrt{n/2}$. Here we study the expectation of $\overline c_{_{{\rm U}(n)}}(OQ)$ with respect to the Haar measure on the orthogonal group ${\rm O}(2n)$.

The main result of this note is the following: 

\begin{theorem} \label{main-thm1}
There exist universal constants $C,c_1,c_2>0$ such that 
$$ \mu \left \{ O \in {\rm O}(2n) \, | \, \exists  {\rm  \ a \ complex \ line \ } L \subset {\mathbb R}^{2n} \ {\rm with} \ {\rm diam}(\pi_L(OQ)) \leq c_1 \sqrt{n} \right \} \leq C{\rm exp}(-c_2 n),$$
where $\mu$ is the unique normalized Haar measure on ${\rm O}(2n)$. 
\end{theorem}

Note that for any rotation $U \in {\rm O}(2n)$, the image $UQ$ contains the Euclidean unit ball and hence for every complex line $L$ one has 
${\rm Area} (\pi_L UQ)  \geq   {\rm diam} (\pi_L UQ)$. 
An immediate corollary from this observation, Theorem~\ref{main-thm1}, and the easily verified fact that 
for every $O \in {\rm O}(2n)$, the complex line $L' := {\rm Span} \{v,J v\}$, where $ v$ is the direction where the minimal-width of $OQ$ is obtained, satisfies ${\rm Area}(\pi_{L'}(OQ)) \leq 4\sqrt{2n}$, is that 
\begin{corollary} \label{Cor-rot-cube}
With the above notations one has
\begin{equation} \label{expectation-rotated-cube}
{\mathbb E}_{\mu} \left ( \overline c_{_{{\rm U}(n)}}(OQ) \right ) \asymp \sqrt{n},
\end{equation}  
where ${\mathbb E}_{\mu}$ stands for the expectation with respect to the Haar measure $\mu$ on ${\rm O}(2n)$, and the symbol $\asymp$ means equality up to universal multiplicative constants.
\end{corollary}

\begin{remark} 
{\rm
We will see below that for every $O \in {\rm O}(2n)$, the quantity $\overline c_{_{{\rm U}(n)}}(OQ)$ is bounded from below by the diameter of the section of the $4n$-dimensional octahedron $B^{4n}_1$ by the subspace 
\begin{equation} \label{the-subspace} L_O = \{ (x,y) \in {\mathbb R}^{2n} \oplus {\mathbb R}^{2n} \ | \ y=O^*JOx \}. \end{equation}
This reduces the above problem of estimating ${\mathbb E}_{\mu} \left ( \overline c_{_{{\rm U}(n)}}(OQ) \right )$ to estimating the diameter of a random section of the octahedron $B^{4n}_1$  with respect to a probability measure $\nu$ on the real Grassmannian $G(4n,2n)$ induced by 
the map $O \mapsto L_O$ from the Haar measure $\mu$ on $O(2n)$.
By duality, the diameter of a section of the octahedron by a linear subspace is equal to the deviation
of the Euclidean ball from the orthogonal subspace with respect the $l_{\infty}$-norm. 
The right order of the minimal deviation from half-dimensional subspaces was found in the remarkable work of Ka\v{s}in~\cite{K}.
For this purpose, he introduced some special measure on the Grassmannian and proved that the approximation of the ball by random subspaces is almost optimal.
In his exposition lecture~\cite{Mit}, 
Mitjagin treated Kashin's work as a result about octahedron sections, which gave a more geometric intuition into it, and rather simplified the proof.   
At about the same time, the diameter of random (this time with respect to the classical Haar measure on the Grassmanian) sections of the octahedron, and more general convex bodies, was studied by Milman~\cite{Mil}; Figiel, Lindenstrauss and Milman~\cite{FLM}; Szarek\cite{S}, and many others with connection with Dvoretzky's theorem (see also~\cite{A-M,G-G, GMT, Gl,Mil1, P-T}, as well as Chapters 5 of~\cite{Pis} and Chapters 5 and 7 of~\cite{AAGM} for more details). 
It turns out that 
random sections of the octahedron $B^{4n}_1$, with respect to the measure $\nu$ on the real Grassmannian $G(4n,2n)$ mentioned above, also have almost optimal diameter. To prove this we use techniques which are now standard in the field. For completeness, all details will be given in Sections~\ref{sec-pre} and~\ref{sec-proof-main-thm} below. 
}
\end{remark}

\noindent {\bf Notations:} The letters $C,c,c_1,c_2, \ldots$  denote positive universal constants that take
 different values from one line to another.  Whenever we write $\alpha  \asymp \beta$, we mean that there exist universal constants $c_1, c_2 > 0$
such that $c_1\alpha \leq  \beta \leq  c_2 \alpha$. For a finite set $V$, denote by $\# V$ the number of elements in $V$. For $a \in {\mathbb R}$ let $[a]$ be its integer part. 
The standard Euclidean inner product and norm on ${\mathbb R}^n$ will be denoted by $\langle \cdot,\cdot \rangle$, and $| \cdot |$, respectively. 
The diameter of a subset $V \subset {\mathbb R}^n$ is denoted by ${\rm diam}(V) = \sup \{ |x-y| \, : \, x,y \in V \}$. For $1 \leq p \leq \infty$, we denote by $l_p^n$ the space ${\mathbb R}^n$ equipped with the norm $\| \cdot \|_p$ given by $\| x \|_p = ( \sum_{j=1}^n \|x_i|^p)^{1/p}$ (where $\|x \|_{\infty} = \max \{ |x_i| \, | \, i=1,\ldots,n\}$), and the unit ball of the space $l_p^n$ is denoted by $B_p^n = \{ x \in {\mathbb R}^n \, | \, \|x \|_p \leq 1 \}$.  
We denote by $S^n$ the unit sphere in ${\mathbb R}^{n+1}$, i.e., $S^n = \{ x \in {\mathbb R}^{n+1} \, | \, |x|^2\ = 1\}$, and by $\sigma_n$ the standard measure on $S^n$. Finally, for a measure space $(X,\mu)$ and a measurable function $\varphi : X \rightarrow {\mathbb R}$ we denote by ${\mathbb E}_\mu \varphi$ the expectation of $\varphi$ with respect to the measure $\mu$. 

\noindent{{\bf Acknowledgments:}} The authors would like to thank the anonymous referee for helpful comments and remarks, and in particular for his/her suggestion to 
elaborate more on the symplectic topology background which partially served as a motivation for the current note. 
%
The second-named author was partially supported by the European Research Council (ERC)
under the European Union's Horizon 2020 research and innovation programme, starting grant No. 637386, and by the ISF grant No. 1274/14.

 \section{Preliminaries} \label{sec-pre}
Here we recall some basic notations and results required for the proof of Theorem~\ref{main-thm1}. 
 
 Let $V$ be a subset of a metric space $(X,\rho)$, and let $\varepsilon >0$. A set ${\mathcal F} \subset V$ is called an $\varepsilon$-net for $V$ if for any $x \in V$ there exist $y \in {\mathcal F}$ such that $\rho(x,y) \leq \varepsilon$. 
It is a well known and easily verified  fact  that for any given set $G$ with $V \subseteq G$, if $\mathcal T$ is a finite $\varepsilon$-net for $G$, then there exists a $2\varepsilon$-net ${\mathcal F}$ of $V$ with 
$ \# {\mathcal F} \leq \# {\mathcal T}$. 

\begin{remark} {\rm From now on, unless stated otherwise, all nets are assumed to be taken with respect to the standard Euclidean metric on the relevant space.}
\end{remark}

Next, fix $n \in {\mathbb N}$ and $0< \theta <1$. We denote by $G^n_{\theta}$ the set $G^n_{\theta} := S^{n-1} \cap \theta \sqrt{n} B_1^n$. 
The following proposition goes back to 
 Ka\v{s}in~\cite{K}. 
The proof below follows Makovoz~\cite{Ma} (cf.~\cite{Sc} and the references therein).
\begin{proposition} \label{e-net-prop}
For every 
$\varepsilon$ such that $8 {\frac {\ln n} n} < \varepsilon < {\frac 1 2}$, there exists a set ${\mathcal T} \subset G^n_{\theta}$ such that $\# {\mathcal T} \leq {\rm exp}({\varepsilon n})$, and which is a $8\theta \sqrt{{\frac {\ln(1/\varepsilon)} {\epsilon} }}$-net for $G^n_{\theta} $.
\end{proposition} 

For the proof of Proposition~\ref{e-net-prop} we shall need the following lemma.
\begin{lemma} \label{lemma1}
For $k,n \in {\mathbb N}$, the set ${\mathcal F}_{k,n} :={\mathbb Z}^n \cap k B_1^n$ is a $\sqrt{k}$-net for the set $k B^n_1$,
and  \begin{equation} \label{bound-e-net} \# {\mathcal F}_{k,n} \leq (2e(1+ n/ k)) )^k. \end{equation} 
\end{lemma}

\begin{proof}[ {\bf Proof of Lemma~\ref{lemma1}}]
Let $x = (x_1,\ldots,x_n) \in kB_1^n$, and set $y_j = [|x_j|]\cdot {\rm sgn}(x_j)$, for $1\leq j \leq n$.
Note that $y=(y_1,\ldots,y_n) \in {\mathcal F}_{k,n} $, and $|x_j-y_j| \leq \min \{1,|x_j|\}$ for any $1 \leq j \leq n$. Thus, $|x -y|^2 = \sum_{j=1}^n |x_j-y_j|^2 \leq  \sum_{j=1}^n |x_j| = k$. This shows that ${\mathcal F}_{k,n}$ is a $\sqrt{k}$-net for $k B^n_1$.
 In order to prove the bound~$(\ref{bound-e-net})$ for the cardinality of ${\mathcal F}_{k,n} $, note that by definition
\begin{eqnarray*} \#{\mathcal F}_{k,n}  & = &  \# \{v \in {\mathbb Z}^n \, | \, \sum_{i=1}^n |v_i| \leq k \} \leq 2^k \# \{  v \in {\mathbb Z}_{+}^{n+1} \, | \, \sum_{i=1}^{n+1} v_i = k \}  \\ & = & 2^k \binom{n+k}{k} \leq 2^k  \Bigl ({\frac {e(n+k)} {k} } \Bigr)^k. \end{eqnarray*}
This completes the proof of the lemma.
\end{proof}

\begin{proof}[ {\bf Proof of Proposition~\ref{e-net-prop}}]
We assume $n>1$ (the case $n=1$ can be checked directly). Set $k=[{\frac {\varepsilon n} {8 \ln (1/\varepsilon)}}]$. Note that since $\varepsilon > 8 {\frac {\ln n} {n} }$, one has that $k \geq 1$.
From Lemma~\ref{lemma1} it follows that  $\theta {{\frac {\sqrt n} {k} }}  {\mathcal F}_{k,n}$ is a $\theta {{\frac {\sqrt n} {k} }} $-net for $\theta \sqrt{n} B_1^n$. From the remark in the beginning of this section and Lemma~\ref{lemma1} we conclude that there is a set ${\mathcal T} \subset G^n_{\theta} \subset \theta \sqrt{n} B_1^n$ which is a $2\theta \sqrt{{\frac {n} {k} }} $-net for $G^n_{\theta}$, and moreover,
\begin{equation*}
\# {\mathcal T} \leq \# {\mathcal F}_{k,n} \leq  \bigl(2 e(1+ n/ k))  \bigr)^k.
\end{equation*}
Finally, from our choice of $\varepsilon$ it follows that $k \geq {\frac {\varepsilon n} {16 \ln (1/\varepsilon)} }$, and hence $2\theta\sqrt{{\frac n k}} \leq 8\theta \sqrt{{\frac {\ln (1/ \varepsilon)} {\varepsilon}}} $, and moreover that $ \bigl (2e (1+n/k) \bigr)^{k/n} \leq e^{\varepsilon}$. This completes the proof of the proposition.  
\end{proof}

We conclude this section with the following well-known result regarding concentration of measure for Lipschitz functions on the sphere (see, e.g.,~\cite{MS}, Section 2 and Appendix V).
\begin{proposition}  \label{prop-about-concentration} Let $f : S^{n-1} \rightarrow {\mathbb R}$ be an $L$-Lipschitz function and set ${\mathbb E}f= \int_{S^{n-1}} f d \sigma_{n-1}$, where $\sigma_{n-1}$ is the standard measure on $S^{n-1}$. Then, $$ \sigma_{n-1} \left ( \{ x \in S^{n-1} \, | \, |f(x)-{\mathbb E}f| \geq t \} \right ) \leq C {\rm exp}(-\kappa t^2n/L^2), $$
where $C,\kappa>0$ are some universal constants.
\end{proposition}

\section{Proof of the Main Theorem} \label{sec-proof-main-thm}
\begin{proof}[ {\bf Proof of Theorem~\ref{main-thm1}}] Let $Q = [-1,1]^{2n} \subset {\mathbb R}^{2n}$. The proof is divided into two steps:

\noindent {\bf Step I ($\varepsilon$-net argument):}
Let $L \subset {\mathbb R}^{2n}$ be a complex line, and $e \in S^{2n-1} \cap L$. Note that the vectors $e$ and $Je$ form an orthogonal basis for $L$, and for every $x \in {\mathbb R}^{2n}$ one has
$$ \pi_L(x) = \langle x,e \rangle e + \langle x,Je \rangle Je.$$
Thus, one has
$$ $$ 
\begin{equation} \label{area-diam-ineq}
\begin{split}
    {\rm diam} (\pi_L (UQ)) & = 2 \max_{x \in Q} \sqrt{| \langle Ux,e \rangle|^2 + | \langle Ux,Je \rangle|^2 } \\
 &\geq  \max_{x \in Q} \max \{| \langle x,U^*e \rangle| , | \langle x,U^*Je \rangle| \} \\
 &=  \max \{\|U^*e\|_1 , \| U^*Je \|_1 \}.
\end{split}
\end{equation}
It follows that for every $U \in {\rm O}(2n)$, the minimum over all complex lines 
satisfies
\begin{equation} \label{est-area}
 \min_{L } {\rm diam} (\pi_L (UQ)) \geq  \min_{v \in S^{2n-1}} \max \{  \|v\|_1, \| U^*JUv \|_1 \}.
\end{equation}
Next, for a given constant $\theta>0$, denote $G_\theta := S^{2n-1} \cap \theta \sqrt{n} B_1^{2n}$, and
\begin{equation} \label{small-diam} {\mathcal A}_{\lambda}  := \{ U \in {\rm O}(2n) \, | \, \exists  {\rm \ a \ complex \ line \ } L \subset{\mathbb R}^{2n} \ {\rm with \ }  {\rm diam}(\pi_L (UQ)) \leq {\lambda} \sqrt{n} \}.\end{equation}
Recall that in order to prove Theorem~\ref{main-thm1}, we need to show that there is a constant $\lambda$ for which the measure of $ {\mathcal A}_{\lambda} \subset {\rm O}(2n) $ is exponentially small, a task to which we now turn.
From~$(\ref{est-area})$ it follows that for any $U \in {\mathcal A}_{\lambda}$ 
one has 
$$ G_{\lambda} \cap U^* JU G_{\lambda} \neq \emptyset.$$
Indeed, if $U \in {\mathcal A}_{\lambda}$, then by~$(\ref{area-diam-ineq})$ one has that $\| U^*e \|_1 \leq \lambda \sqrt{n}$ and $\| (U^*JU)U^*e \|_1 \leq \lambda \sqrt{n}$, so $z := U^*e_1 \in G_{\lambda}$ and $U^*JUz \in G_{\lambda}$. 
Hence, we conclude that
$$ {\mathcal A}_{\lambda}  \subseteq \{ U \in {\rm O}(2n) \, | \,  G_{\lambda} \cap U^* JU G_{\lambda} \neq \emptyset \}.$$ 
Next, let ${\mathcal F}$ be a $\delta$-net for $G_{\lambda}$ for some $\delta >0$.  
For any $U \in {\mathcal A}_{\lambda}$ there exists $x \in G_{\lambda} \cap U^* JU G_{\lambda}$, and 
 $y \in {\mathcal F}$ for which $|y - x| \leq \delta$. Thus, one has
$$   \| U^*JUy \|_1 \leq \| U^*JUx \|_1 + \| U^*JU (y-x) \|_1  \leq {\lambda} \sqrt{n} + \sqrt{2n} | U^*JU (y-x) | \leq \sqrt{n}({\lambda} + \sqrt{2}\delta).$$ 
It follows that 
\begin{equation} \label{estimate-special-rotations} {\mathcal A}_{\lambda} 
\subseteq \bigcup_{y \in {\mathcal F}}  \left \{ U \in {\rm O}(2n) \, | \,   U^* JU y \in G_{{\lambda}+ \sqrt{2}\delta} \right \}.  \end{equation}
From~$(\ref{estimate-special-rotations})$ and Proposition~\ref{e-net-prop} from Section~\ref{sec-pre} it follows that for every ${\lambda}>0$
\begin{equation} \label{estimate-special-rotations1}
\begin{split}
\mu({\mathcal A}_{\lambda}) & \leq  \sum_{y \in {\mathcal F}} \mu \{ U \in O(2n) \, | \, U^*JUy \in G_{{\lambda}+\sqrt{2}\delta} \} \\
 & \leq \exp({2\varepsilon n}) \sup_{y \in S^{2n-1}} \mu \{ U \in O(2n) \, | \, U^*JUy \in G_{{\lambda}+\sqrt{2}\delta} \},
\end{split}
\end{equation}
where $ 8 {\frac {\ln(2n)} {2n} } < \varepsilon < {\frac 1 2}$, and $\delta = 8{\lambda} \sqrt{{\frac {\ln(1/ \varepsilon)} {\varepsilon} }} $. 

\noindent {\bf Step II (concentration of measure):} For $y \in S^{2n-1}$ let $\nu_y$ be the push-forward  measure on $S^{2n-1}$ induced by the Haar measure $\mu$ on ${\rm O}(2n)$ through the map $f : O(2n) \rightarrow S^{2n-1}$ defined by $U \mapsto U^*JUy$. Using the measure $\nu_y$, we can rewrite inequality~$(\ref{estimate-special-rotations1})$ as 
\begin{equation} \label{estimate-special-rotations2}  \mu({\mathcal A}_{\lambda}) \leq \exp({2\varepsilon n}) \sup_{y \in S^{2n-1}} \nu_y (G_{{\lambda}+\sqrt{2}\delta}) = \exp({2\varepsilon n}) \sup_{y \in S^{2n-1}} \nu_y  \{ x \in S^{2n-1} \, | \, \|x\|_1 \leq \sqrt{n}({\lambda}+\sqrt{2} \delta)\}. \end{equation} 
Note that if $V \in {\rm O}(2n)$ preserves $y$, i.e., $Vy=y$, then 
$$ V(f(U))=V(U^*JUy) = (UV^*)^* J (UV^*)(Vy)= f(UV^*).$$ Thus, the measure $\nu_y$ is invariant under any rotation in ${\rm O}(2n)$ that preserves $y$. Note also that for any $y \in S^{2n-1}$ one has 
$$ \langle U^*JUy,y \rangle = \langle JUy,Uy \rangle = 0.$$
This means that $\nu_y$ is supported on $S^{2n-1} \cap \{y \}^{\perp}$, and hence we conclude that $\nu_y$ is the standard normalized measure on $S^{2n-1} \cap \{y \}^{\perp}$.

Next, let $S_y = S^{2n-1} \cap \{y\}^{\perp}$. For $x \in S_y$ set $\varphi(x)= \|x\|_1$. Note that $\varphi$ is a Lipschitz function on $S_y$ with Lipschitz constant $\|\varphi\|_{\rm Lip} \leq \sqrt{2n}$. 
Using a concentration of measure argument (see Proposition~\ref{prop-about-concentration} above), we conclude that for any $\alpha >0$ 
\begin{equation} \label{estimate-123} \nu_y \{ x \in S_y \, | \, \varphi(x) < {\mathbb E}_{\nu_y} \varphi - \alpha \sqrt{n} \} \leq C{\rm exp}(-\kappa^2 \alpha^2 n^2 / \| \varphi \|^2_{\rm Lip}) \leq  C{\rm exp} (-\kappa^2 \alpha^2 n), \end{equation}
for some universal constants $C$ and $\kappa$.

Our next step is to estimate the expectation ${\mathbb E}_{\nu_y} \varphi $ that appear in~$(\ref{estimate-123})$. For this purpose let us take some orthogonal basis $\{ z_1,\ldots, z_{2n-1}\}$ of the subspace $L= \{ y\}^{\perp} \subset {\mathbb R}^{2n}$. For $1 \leq j \leq 2n$, denote by $w_j$ the vector $w_j = (z_1(j),\ldots,z_{2n-1}(j))$, where $z_k(j)$ stands for the $j^{\rm th}$ coordinate of the vector $z_k$. Then, the measure $\nu_y$, which is the standard normalized Lebesgue measure on $S^{2n-1} \cap \{y \}^{\perp}$, can be described as the image of the normalized Lebesgue measure $\sigma_{2n-2}$ of $S^{2n-2}$ under the map
$$ S^{2n-2} \ni a = (a_1,\ldots,a_{2n-1}) \mapsto \sum_{k=1}^{2n-1} a_k z_k = (\langle a ,w_1 \rangle, \langle a,w_2 \rangle,\ldots \langle a, w_{2n} \rangle ) \in S_y.$$
Consequently, 
$$ {\mathbb E}_{\nu_y} \varphi  =  {\mathbb E}_{\sigma_{2n-2}} (a \mapsto \sum_{j=1}^{2n} | \langle a,w_j \rangle | ) \geq {\frac 1 {\sqrt {2n-1}} } \sqrt {\frac 2 {\pi}} \sum_{j=1}^{2n} |w_j|.$$
Since $\{z_1,\ldots,z_{2n-1},y\}$ is a basis of ${\mathbb R}^{2n}$, one has that $|w_j|^2 +y_j^2 =1$ and hence 
$$ {\mathbb E}_{\nu_y} \varphi = {\frac 1 {\sqrt{2n-1}}} \sqrt{\frac 2 {\pi}} \sum_{j=1}^{2n} \sqrt{1-y_j^2} \geq  {\frac 1 {\sqrt{2n-1}}} \sqrt{\frac 2 {\pi}} \, (2n-1) \geq {\frac 1 2 }  \sqrt{n}.$$
Thus, from inequality~$(\ref{estimate-123})$ with $\alpha = {\frac 1 4}$ we conclude that
\begin{equation} \label{estimate-1234}  \nu_y \{ x \in S_y \, | \, \varphi(x) < {\frac 1 4} \sqrt{n} \} \leq \nu_y \{ x \in S_y \, | \, \varphi(x) <  {\mathbb E}_{\nu_y} \varphi  - {\frac 1 4} \sqrt{n} \} \leq C{\rm exp} (-{\frac {\kappa^2n} {16}}). \end{equation}
In other words, for any $\theta \leq {\frac 1 4}$ and any $y \in S^{2n-1}$ one has that 
$$ \nu_y (G_\theta) \leq C{\rm exp} (-{\frac {\kappa^2n} {16}}),$$
for some constant $\kappa$. Thus, for every ${\lambda}$ such that $\lambda+\sqrt{2} \delta \leq 1/4$, we conclude by~$(\ref{estimate-special-rotations2})$ that 
$$\mu({\mathcal A}_{\lambda}) \leq C{\rm exp}({2n\varepsilon}) \cdot {\rm exp} \bigl(-{\frac {\kappa^2 n} {16}} \bigr). $$
To complete the proof of the Theorem it is enough to take  $\varepsilon = \kappa^2/64$, and $\lambda$ which satisfies the inequality $\lambda \Bigl (1+16 \sqrt{{\frac {\ln(1/ \varepsilon)} {\varepsilon}}} \Bigr) \leq 1/4$. 
\end{proof}

\section*{Appendix} 

Here we provide some background from symplectic topology which partially served as a motivation for the current paper. For more 
detailed information on symplectic topology we refer the reader e.g., to the books~\cite{HZ,McSa} and the references therein. 

A symplectic vector space is a pair $(V,\omega)$, consisting of a finite-dimensional vector space and a non-degenerate skew-symmetric bilinear form $\omega$, 
called the symplectic structure. 
The group of linear transformations which preserve $\omega$ is denoted by ${\rm Sp}(V,\omega)$. 
The archetypal example of a symplectic vector space is the Euclidean space ${\mathbb R}^{2n}$ equipped with the skew-symmetric bilinear form 
$ \omega$ 
which is the imaginary part of the standard Hermitian inner product in $ {\mathbb R}^{2n} \simeq {\mathbb C}^n$. 
More precisely, if $\{x_1,\ldots,x_n,y_1,\ldots,y_n\}$ stands for the standard basis of $ {\mathbb R}^{2n}$, then $\omega(x_i,x_j) = \omega(y_i,y_j) =0$, and $\omega(x_i,y_j)=\delta_{ij}$. In this case the group of linear symplectomorphisms is usually denoted by ${\rm Sp}(2n)$. 
More generally, the group of diffeomorphisms $\varphi$ of ${\mathbb R}^{2n}$ which preserve the symplectic structure, i.e., when the differential $d \varphi$ at each point is a linear symplectic map, is called the group of symplectomorphisms of ${\mathbb R}^{2n}$, and is denoted by ${\rm Symp}({\mathbb R}^{2n},\omega)$. In the spirit of Klein's Erlangen program, symplectic geometry can be defined as the study of transformations which preserves the symplectic structure. 
We remark that already in the linear case, the geometry of a skew-symmetric bilinear form is very different from that of a symmetric form, e.g., there is no natural notion of distance or angle between two vectors. We further remark that symplectic vector spaces, and more generally symplectic manifolds,     
 provide a natural setting for Hamiltonian dynamics, as the evolution of a Hamiltonian system is known to preserve the symplectic form  (see, e.g.,~\cite{HZ}). Historically, this is the main motivation to study symplectic geometry. 

In sharp contrast with Riemannian geometry where, e.g., curvature is an
obstruction for two manifolds to be locally isometric, in the realm of symplectic geometry it is known that there are no local invariants (Darboux's theorem).
Moreover, unlike the Riemannian setting, a symplectic structure has a very rich group of automorphisms. More precisely, the group of 
symplectomorphisms is
 an infinite-dimensional Lie group.
The first results distinguishing (non-linear) symplectomorphisms from volume preserving transformations were discovered only in the 1980s. 
The most striking difference between the category of volume preserving transformations and the category of symplectomorphisms was demonstrated by Gromov~\cite{Gr} in his famous non-squeezing theorem. This theorem asserts that if $r < 1$, there is no symplectomorphism $\psi$ of ${\mathbb R}^{2n}$ which maps the open unit ball $B^{2n}(1)$ into the open cylinder $Z^{2n}(r) = B^2(r) \times {\mathbb C}^{n-1}$. 
This result paved the way to the introduction of global symplectic
invariants, called symplectic capacities, which are significantly differ from any
volume related invariants, and roughly speaking measure
the symplectic size of a set (see e.g.,~\cite{CHLS}, for the precise definition and further discussion).
%
%
%
%
Two examples, defined for open subsets of ${\mathbb R}^{2n}$, are the Gromov radius  $\underline c({\mathcal U}) = \sup \{ \pi r^2 \, : \, B^{2n}(r) \stackrel{s}{\hookrightarrow} {\mathcal U} \}$,
 and the cylindrical capacity $\overline c({\mathcal U}) = \inf \{ \pi r^2 \, : \, {\mathcal U}  \stackrel{s}{\hookrightarrow} Z^{2n}(r) \}$.  Here $ \stackrel{s}{\hookrightarrow} $  stands for symplectic
embedding.
 
Shortly after Gromov's work~\cite{Gr} many other symplectic capacities were constructed,
reflecting different geometrical and dynamical properties. Nowadays, these invariants play an important role
in symplectic geometry, and their properties, interrelations, and applications to symplectic
topology and Hamiltonian dynamics are intensively studied (see e.g.,~\cite{CHLS}). 
However, in spite of the rapidly accumulating knowledge regarding symplectic capacities, they
are usually notoriously difficult to compute, and there are very few general methods to effectively estimate them, even within the class of convex domains in ${\mathbb R}^{2n}$ (we refer the reader to~\cite{O} for a survey of some known results and open questions regarding symplectic measurements of convex sets in ${\mathbb R}^{2n}$).
In particular, a long standing central question is whether all symplectic capacities coincide on the class of convex bodies in ${\mathbb R}^{2n}$ (see, e.g., Section 5 in~\cite{O}). Recently, the authors proved that for centrally symmetric convex bodies, several symplectic capacities, including the Ekeland-Hofer-Zehnder capacity $c_{_{\rm EHZ}}$, spectral capacities, the cylindrical capacity $\overline c$, and its linearized version $c_{_{{\rm Sp}(2n)}}$ given in~$(\ref{def-linearized-cilindrical-capacity})$, are all equivalent up to an absolute constant.
More precisely,  the following was proved in~\cite{GO}.
\begin{theorem} \label{GO-thm}
For every centrally symmetric convex body $K \subset {\mathbb R}^{2n}$ 
\begin{equation*}
{\frac {1} {\|J\|_{K^{\circ} \rightarrow K}} } \leq c_{_{{\rm EHZ}}}(K) \leq  \overline c(K) \leq   \overline c_{_{{\rm Sp}(2n)}} (K) \leq {\frac {4} {\|J\|_{K^{\circ} \rightarrow K}} } ,
\end{equation*}
where $\|J \|_{K^{\circ} \rightarrow K}$ is the operator norm of the complex structure $J$, when the latter is 
considered as a linear map between the normed spaces $J : ( {\mathbb R}^{2n}, \| \cdot \|_{K^{\circ}} ) \rightarrow  ( {\mathbb R}^{2n}, \| \cdot \|_{K} ).$
\end{theorem}

Theorem~\ref{GO-thm} implies, in particular, that despite the non-linear nature of the Ekeland-Hofer-Zehnder capacity $c_{_{\rm EHZ}}$, and  the cylindrical capacity $\overline c$ (both, by definition, are invariant under non-linear symplectomorphisms), for centrally symmetric convex bodies they are asymptotically equivalent to a linear invariant: the linearized cylindrical capacity $\overline c_{_{{\rm Sp}(2n)}} $. 
Motivated by the comparison between the capacities $\overline c$ and $\overline c_{_{{\rm Sp}(2n)}}$ in Theorem~\ref{GO-thm}, it is natural to introduce and study 
 the following geometric quantity:
\begin{equation} \label{new-geom-quant} \overline c_{_{{\rm G}}}({K})  = \inf_{g \in {\rm G}} {\rm Area} \bigl (\pi( g( {K})) \bigr ), \end{equation}
where $K$ lies in the class of convex domains of ${\mathbb R}^{2n} \simeq {\mathbb C}^n$ (or possibly, some other class of bodies), $\pi$ is the orthogonal projection to the complex line $E = \{ z \in {\mathbb C}^n \, | \,  z_j = 0 \ {\rm for} \ j \neq 1  \}$, and  $G$ is some group of transformations of ${\mathbb R}^{2n}$.
%
%
One possible choice  is to take the group $G$ in~$(\ref{new-geom-quant})$ to be the unitary group ${\rm U}(n)$, which is the maximal compact subgroup of ${\rm Sp}(2n)$.
In this case it is not hard to check (by looking at linear symplectic images of the cylinder $Z^{2n}(1)$) that the cylindrical capacity $\overline c$ is not asymptotically equivalent to $ \overline c_{_{{\rm U}(n)}}$.
 Still, one can ask if these two quantities are asymptotically equivalent on average. More precisely, 

\begin{ques} \label{Sp-vs-U} Is it true that for every convex body $K \subset {\mathbb R}^{2n}$ one has
$$ {\mathbb E}_{\mu} \left (\overline c(OK) \right ) \asymp  {\mathbb E}_{\mu} \bigl( \overline c_{_{{\rm U}(n)}} (OK) \bigr) \, ?,$$
where $\mu$ is the 
Haar measure on the orthogonal group ${\rm O}(2n)$.
\end{ques}

The answer to Question~\ref{Sp-vs-U} is negative. A counterexample is given by the standard cube $Q = [-1,1]^{2n}$ in ${\mathbb R}^{2n}$.
We remark that the quantity $ {\mathbb E}_{\mu} \bigl( \overline c_{_{{\rm U}(n)}} (OQ) \bigr) $ is the main objects of interest of the current paper. 
To be more precise, we turn now to the following proposition, which is a direct corollary of Theorem~\ref{GO-thm}, and might be of independent interest. 
For completeness, we shall give a proof below.

\begin{proposition} \label{HZ-capacity-of-rotated cube}
For the standard cube $Q = [-1,1]^{2n} \subset {\mathbb R}^{2n}$ one has 
\begin{equation*} \label{expectation-capacities}
{\mathbb E}_{\mu} \left ( c_{_{\rm EHZ}}(OQ) \right ) \asymp {\mathbb E}_{\mu} \left ( \overline c(OQ) \right ) \asymp {\mathbb E}_{\mu} \left ( \overline c_{_{{\rm Sp}(2n)}}(OQ) \right )  \asymp \sqrt{{\frac n {\ln n}}}, 
\end{equation*} 
where $\mu$ is the 
Haar measure on the orthogonal group ${\rm O}(2n)$.
\end{proposition}

Note that the combination of the main result of the current paper (in particular, Corollary~\ref{Cor-rot-cube}) with 
 Proposition~\ref{HZ-capacity-of-rotated cube} above gives a negative answer to Question~\ref{Sp-vs-U}, and thus further emphasizes the difference between the symplectic and complex structures on ${\mathbb R}^{2n} \simeq {\mathbb C}^n$.

\begin{proof}[{\bf Proof of Proposition~\ref{HZ-capacity-of-rotated cube}}]
Note that by definition one has that 
$$ \| J \|_{(OQ)^{\circ} \rightarrow (OQ)}  =  \max_{x \in (OQ)^{\circ}} \|Jx\|_{OQ}  = \max_{x \in B_1^{2n}} \|O^*JOx\|_{{\infty}}   =  \max_{i=1,\ldots,2n}  \|O^*JOe_i \|_{{\infty}},
$$
where $\{e_i\}_{i=1}^{2n}$ stands for the standard basis of ${\mathbb R}^{2n}$.
It follows from Step II of the proof of Theorem~\ref{main-thm1} above that for a random rotation $O \in {\rm O}(2n)$, the vector $O^*JOe_i$ is uniformly distributed on $S^{2n-2} \simeq S^{2n-1}  \cap \{e_i \}^{\perp}$ with respect to the standard normalized measure  ${\sigma_{_{2n-2}}}$ on $S^{2n-2}$. 
The distribution of the $l_{\infty}^k$-norm on the sphere $S^{k-1}$ is well-studied, and in particular one has  (see e.g., Sections 5.7 and 7 in~\cite{MS}) that for every $e_i$ 
\begin{equation} \label{estimate-on-expectation} {\mathbb E}_{\mu} \left (   \|(O^*JOe_i)\|_{\infty} \right ) \asymp \sqrt{{\tfrac {\ln n} {n}}}, \end{equation} and 
\begin{equation} \label{estimate-on-probability} {\mathbb P}_{\mu} \left \{   (  \|(O^*JOe_i)\|_{\infty} -  {\mathbb E}_{\mu} \left (   \|(O^*JOe_i)\|_{\infty} \right ) > t  \right \}  \leq c_1 {\rm exp}(-c_2 t^2n), \end{equation}
for some universal constants $c_1,c_2 >0$.
From~$(\ref{estimate-on-expectation})$ and~$(\ref{estimate-on-probability})$ it immediately follows that 
\begin{equation} \label{expectation-of-norm-J} {\mathbb E}_{\mu} \left (   \| J \|_{(OQ)^{\circ} \rightarrow (OQ)} \right ) \asymp \sqrt{{\tfrac {\ln n} {n}}}. \end{equation}
Moreover, one has that for some universal constants $c_3,c_4 >0$,
\begin{equation} \label{estimate-on-probability-of-norm-J} {\mathbb P}_{\mu} \bigl \{  (  \| J \|_{(OQ)^{\circ} \rightarrow (OQ)}  \leq  c_3 \sqrt{{\tfrac {\ln n} {n}}}   \bigr \}  \leq {\frac {c_4} n}. \end{equation}
Indeed, from the above it follows that
\begin{equation*} 
{\mathbb P}_{\mu} \bigl \{    \| J \|_{(OQ)^{\circ} \rightarrow (OQ)}  \leq  t   \bigr \}    \leq  {\mathbb P}_{\mu} \bigl \{  (  \|(O^*JOe_1)\|_{\infty} \leq  t   \bigr \}  = {\mathbb P}_{\sigma_{_{2n-2}}} \bigl \{    \| v \|_{\infty} \leq  t   \bigr \}.  
\end{equation*}
Using the standard Gaussian probability measure ${\gamma_{_{2n-1}}}$ on ${\mathbb R}^{2n-1}$, one can further estimate 
\begin{eqnarray*} 
 {\mathbb P}_{\sigma_{_{2n-2}}} \bigl \{    \| v \|_{\infty} \leq  t   \bigr \}  & = & {\gamma_{_{2n-1}}} \bigl \{    \| g \|_{\infty} \leq  t \|g\|_2   \bigr \}  \\ & \leq & {\gamma_{_{2n-1}}} \bigl \{    \| g \|_{\infty} \leq   2 {\sqrt{2n-1}}   t    \bigr \}  + {\gamma_{_{2n-1}}} \bigl \{    \| g \|_{2} \geq 2 {{\sqrt{2n-1}}}   \bigr \},
\end{eqnarray*}
where $g$ is a Gaussian vector in ${\mathbb R}^{2n-1}$ with independent standard Gaussian coordinates. One can directly check that~$(\ref{estimate-on-probability-of-norm-J})$ now follows from the above inequalities, and the following standard estimates for the Gaussian probability measure $\gamma_{_{k}}$ on ${\mathbb R}^{k}$, and $0< \varepsilon <1$: $$\gamma_{_{k}} \bigl \{    \| g \|_{\infty} \leq  \alpha  \bigr \} \leq [1- {\sqrt{\tfrac {2} {\pi} }} {\tfrac {{\rm exp}(-\alpha^2/2)} {\alpha} } ]^k, \ {\rm and} \ \gamma_{_{k}} \left  \{ x \in {\mathbb R}^k \, | \,  \|g\|_2^2 \geq {\tfrac {k} {(1-\varepsilon)}} \right \} \leq {\rm exp}(-\varepsilon^2k/4).$$
Taking into account the fact that $  {\frac 1 {\sqrt {2n}}} \leq  \| J \|_{(OQ)^{\circ} \rightarrow (OQ)} \leq 1$, we conclude from~$(\ref{expectation-of-norm-J})$ and~$(\ref{estimate-on-probability-of-norm-J})$ above that
\begin{equation*} \label{expectation-of-1-over-norm-J}   {\mathbb E}_{\mu} \bigl ( (  {  \| J \|_{(OQ)^{\circ} \rightarrow (OQ)}} )^{-1} \bigr) 
\asymp \sqrt{{\tfrac {n} {\ln n}}}. \end{equation*}
Together with Theorem~\ref{GO-thm}, this completes the proof of Proposition~\ref{HZ-capacity-of-rotated cube}. 
\end{proof}


 \noindent
Efim D. Gluskin \\
School of Mathematical Sciences \\
Tel Aviv University, Tel Aviv 69978, Israel  \\
{\it e-mail}: gluskin@post.tau.ac.il \\

\vskip -5pt

 \noindent
Yaron Ostrover \\
School of Mathematical Sciences \\
Tel Aviv University, Tel Aviv 69978, Israel  \\
{\it e-mail}: ostrover@post.tau.ac.il \\

\end{document}